\documentclass{amsart}
\usepackage{graphicx}
\theoremstyle{plain}
  \newtheorem{thm}{Theorem}
  \newtheorem{lem}{Lemma}
\theoremstyle{definition}
  \newtheorem{defn}{Definition}
  
\theoremstyle{remark}
  \newtheorem{rem}{Remark}
\theoremstyle{proposition}
  \newtheorem{pr}{Proposition}
\theoremstyle{proof}
  
\theoremstyle{property}
  
\theoremstyle{corollary}
  \newtheorem{cor}{Corollary}
\newcommand{\xqedhere}[2]{%
\rlap{\hbox to#1{\hfil\llap{\ensuremath{#2}}}}}

\begin{document}

\title[L\'{e}vy constant and Chernoff-type estimate]{Random Continued fractions: L\'{e}vy constant and Chernoff-type estimate}
\author{Lulu Fang$^1$ Min Wu$^1$ Narn-Rueih Shieh$^2$ and Bing Li$^{1, *}$}
\address{1 Department of Mathematics, South China University of Technology, Guangzhou, 510641, P.R. China}
\email{fanglulu1230@163.com, wumin@scut.edu.cn, libing0826@gmail.com}
\address{2 Department of Mathematics, Honorary Faculty, National Taiwan University, Taipei, 10617, Taiwan}
\email{shiehnr@ntu.edu.tw}
\thanks{* Corresponding author}

\begin{abstract}
  Given a stochastic process $\{A_n, n \geq 1\}$ taking values in natural numbers, the random continued fractions is defined as $[A_1, A_2, \cdots, A_n, \cdots]$ analogue to the continued fraction expansion of real numbers. Assume that  $\{A_n, n \geq 1\}$ is ergodic and the expectation $E(\log A_1) < \infty$, we give a L\'{e}vy-type metric theorem which covers that of real case presented by L\'{e}vy in 1929. Moreover, a corresponding Chernoff-type estimate is obtained under the conditions $\{A_n, n \geq 1\}$ is $\psi$-mixing and for each $0< t< 1$, $E(A_1^t) < \infty$.
\end{abstract}
\maketitle

\section{Introduction}
Let $T: [0,1) \longrightarrow [0,1)$ be the Gauss transformation defined by
\[
Tx =
\begin{cases}
\frac{1}{x} - \left\lfloor\frac{1}{x}\right\rfloor &\text{if $x \in (0,1)$} \\
0  & \text{if $x=0$},
\end{cases}
\]
where $\lfloor x\rfloor$ denotes the greatest integer not exceeding $x$. Then every $x \in [0,1)$ can be written as the following recursive form
\begin{equation}\label{continued fraction expansion}
x = \dfrac{1}{a_1(x) +\dfrac{1}{a_2(x) + \ddots +\dfrac{1}{a_n(x)+ T^nx}}},
\end{equation}
where $a_1(x) = \lfloor\frac{1}{x}\rfloor$ and $a_n(x) = a_1(T^{n -1}x)$ for all $n \geq 2$ are called the partial quotients of $x$ . If there exists $n \in \mathbb{N}$ such that $T^{n}x =0$, we say that the form (\ref{continued fraction expansion}) is a finite continued fraction expansion of $x$ denoted by $[a_1(x),a_2(x),\cdots,a_n(x)]$. Otherwise, the form (\ref{continued fraction expansion}) is said to be an infinite continued fraction expansion of $x$ denoted by $[a_1(x),a_2(x),\cdots,a_n(x),\cdots]$. For more details about continued fractions, we refer the reader to \cite{les79},\cite{les80},\cite{les81},\cite{les76} and the references therein.

It is well-known that for every real number $x \in [0,1)$, there corresponds a continued fraction expansion. Moreover, the expansion is finite if $x$ is rational and the expansion is infinite if $x$ is irrational (see \cite{les76}, Theorem 14).

For any $x \in [0,1)$ and $n \geq 1$, we define the truncated continued fraction
\[
\frac{p_n(x)}{q_n(x)}:= [a_1(x), a_2(x), \cdots, a_n(x)],
\]
where $p_n(x)$ and $q_n(x)$ are relatively prime integers. We say $p_n(x)/q_n(x)$ the $n$-th convergent of the continued fraction expansion of $x$. Clearly these convergents are rational numbers and $p_n(x)/q_n(x) \rightarrow x$ as $n \rightarrow \infty$ for all $x \in [0,1)$. More precisely, for any $x \in [0,1)$ and $n \geq 1$,
\begin{equation*}
\frac{1}{2q_{n+1}^2(x)}\leq \left|x-\frac{p_n(x)}{q_n(x)}\right| \leq \frac{1}{q_n^2(x)}.
\end{equation*}
This is to say the speed of $p_n(x)/q_n(x)$ approximating to $x$ is dominated by $q_n^{-2}(x)$. So the denominator of the $n$-th convergent $q_n(x)$ plays an important role in the problem of Diophantine approximation. Concerning pointwise asymptotic behaviour of the sequence $\{q_n(x), n \geq 1\}$, L\'{e}vy \cite{les82} obtained the following theorem.

\begin{thm}[\cite{les82}]\label{real Levy}
For $\mathcal{L}$-almost every $x \in [0, 1)$,
\begin{equation}\label{levy}
\lim\limits_{n \to \infty}\frac{1}{n}\log q_n(x) =  \frac{\pi^2}{12\log2},
\end{equation}
where $\mathcal{L}$ denotes the Lebesgue measure.
\end{thm}

If $\lim\limits_{n \to \infty}\frac{1}{n}\log q_n(x)$ exists, such limit is called the L\'{e}vy constant of $x$, which has been studied by many mathematicians, see \cite{les94}, \cite{les88}, \cite{les95} and \cite{les96}. Some other limit theorems for $\{q_n(x), n \geq 1\}$ in continued fraction expansions of real numbers have been extensively investigated. For instance, the central limit theorem for $\{q_n(x), n \geq 1\}$ was given in \cite{les89} and the corresponding rate of convergence was studied by G. Misevi\v{c}ius \cite{les92}, the law of the iterated logarithm for $\{q_n(x), n \geq 1\}$ was provided by \cite{les91}, \cite{les90}.

Notice that $\frac{1}{n}\log q_n(x)$ converges to $\frac{\pi^2}{12\log2}$ in probability, a natural question is what the rate of convergence in (\ref{levy}) is. We show that such rate is at most exponential by the following theorem which is a Chernoff-type estimate for $\frac{1}{n}\log q_n(x)$.

\begin{thm}\label{Levy guji}
For any $\delta > 0$, there exist $N > 0$, $B > 0$, $\alpha > 0$ such that for all $n \geq N$,
\begin{equation*}
\mathcal{L}\left(\left|\dfrac{1}{n}\log q_n - \dfrac{\pi^2}{12\log2}\right| \geq \delta\right) \leq B\exp(-\alpha n).
\end{equation*}
\end{thm}

Furthermore, we will prove a more general version of Theorem \ref{Levy guji} under the setting of so-called random continued fractions (see Theorem \ref{estimate theorem} below). Now we are ready to state this setting. Let $(\Omega, \mathcal{F}, P)$ be a probability space and $\{A_n, n \geq 1\}$ be a stochastic  process defined on $(\Omega, \mathcal{F}, P)$, taking values in the measurable space $(\{1, 2, \cdots, n, \cdots \}, \mathcal{C})$, where $\mathcal{C}$ is the power set of $\{1, 2, \cdots, n, \cdots \}$. For any $\omega \in \Omega$, we define
\begin{align}\label{random expansion}
X(\omega) := [A_1(\omega),A_2(\omega), \cdots, A_n(\omega),\cdots]=\dfrac{1}{A_1(\omega) +\dfrac{1}{A_2(\omega) + \ddots +\dfrac{1}{A_n(\omega)+ \ddots}}}.
\end{align}
Just as for infinite series, the question naturally arises as to whether the right-hand side of (\ref{random expansion}) formally defined is convergent. Fortunately, the convergence of the right-hand side of (\ref{random expansion}) is assured by Theorem 10 in \cite{les76} and Proposition \ref{h} in Section 2 guarantees that $X$ is a random variable taking values in the unit interval $(0,1)$. We say that $X$ is a random continued fractions generated by the stochastic  process $\{A_n, n \geq 1\}$. Different models of random continued fractions have been considered in literature for instance \cite{LP12}, \cite{Lorentzen13}, \cite{Lyons00},\cite{SSU01} and the references therein.

For any $\omega \in \Omega$ and $n \in \mathbb{N}$ , the two quantities $P_n$ and $Q_n$ are defined as
\[
\dfrac{P_n(\omega)}{Q_n(\omega)}:=[A_1(\omega),A_2(\omega),\cdots A_n(\omega)]= \dfrac{1}{A_1(\omega) +\dfrac{1}{A_2(\omega) + \ddots +\dfrac{1}{A_n(\omega)}}},
\]
where $P_n(\omega)$ and $Q_n(\omega)$ are relatively prime integers. We say that $P_n/Q_n$ is the $n$-th convergent of $X$. Moreover, if the limit $\lim\limits_{n \to \infty}\frac{1}{n}\log Q_n$ exists, such limit is said to be the L\'{e}vy constant of $X$.

In this paper, we study a limit theorem for $\{Q_n, n \geq 1\}$ and a corresponding Chernoff-type estimate. We use the notation $E(\xi)$ to denote the expectation of a random variable $\xi$. The definitions of ergodic and $\psi$-mixing stochastic process and some related ones as well will be given in Section 2. Firstly we obtain the following L\'{e}vy-type metric theorem which states that the L\'{e}vy constant of random continued fractions exists and equals a constant almost surely (a.s.).

\begin{thm}\label{random constant}
Let $\{A_n, n \geq 1\}$ be a stochastic process taking values in natural numbers.
If $\{A_n, n \geq 1\}$ is  ergodic and $E(\log A_1)< \infty$, then
\begin{equation}\label{random constant equation}
\lim\limits_{n \to \infty}\dfrac{1}{n}\log Q_n = -\int_{\Omega}\log X dP,\ \ \ \ P-a.s.
\end{equation}
\end{thm}

The convergence in (\ref{random constant equation}) implies that $\{\frac{1}{n}\log Q_n, n \geq 1\}$ converges to $-\int_{\Omega}\log X dP$ in probability, that is, for any $\delta > 0$,
\begin{align}\label{in probability}
\lim\limits_{n \to \infty}P\left(\left|\dfrac{1}{n}\log Q_n + \int_{\Omega}\log X dP\right| \geq \delta\right) =0.
\end{align}
A natural question is what the rate of convergence in (\ref{in probability}) is. We prove that such rate is at most exponential by the following theorem which is a type of Chernoff estimate on the convergence of $\{\frac{1}{n}\log Q_n, n \geq 1\}$.

\begin{thm}\label{estimate theorem}
Let $\{A_n, n \geq 1\}$ be a stochastic process taking values in natural numbers.
If $\{A_n, n \geq 1\}$ is $\psi$-mixing and $E(A_1^t)< \infty$ for each $0< t< 1$, then for any $\delta > 0$, there exist $N > 0$, $B > 0$, $\alpha > 0$ such that for all $n \geq N$, we have
\[
P\left(\left|\dfrac{1}{n}\log Q_n + \int_{\Omega}\log X dP\right| \geq \delta\right) \leq B\exp(-\alpha n).
\]
\end{thm}

 Now we turn to showing the continued fraction expansion of real numbers can be regarded as a special case of random continued fractions. Let $\mathbb{I} = [0,1)\cap\mathbb{Q}^c$ and $\mathcal{B}$ be the Borel $\sigma$-algebra on $\mathbb{I}$, where $\mathbb{Q}^c$ denotes the set of irrational numbers. If $\nu$ is a probability measure on the measurable space $(\mathbb{I}, \mathcal{B})$, then the sequence of partial quotients $\{a_n, n \geq 1\}$ with respect to (w.r.t.) the probability measure $\nu$ forms a stochastic process taking values in natural numbers. Furthermore, if $\nu$ is an invariant measure w.r.t. the Gauss transformation $T$, then the stochastic process $\{a_n, n \geq 1\}$ is stationary.
In 1800, Gauss found such an invariant measure called the Gauss measure, which is given by
\[
\mu(A) = \frac{1}{\log 2}\int_A \frac{1}{1+x}dx
\]
for any Borel set $A \subseteq [0, 1)$ and is equivalent to the Lebesgue measure $\mathcal{L}$. Moreover, the dynamic systems $(\mathbb{I},\mathcal{B},\mu,T)$ is an ergodic system (see \cite{les79}, Theorem 3.5.1) and the partial quotients sequence $\{a_n, n \geq 1\}$ is $\psi$-mixing with exponential rate (see \cite{les84}, Lemma 2.1). If we choose $(\Omega, \mathcal{F}, P)$ = $(\mathbb{I}, \mathcal{B},\mu)$, then $a_n$, $p_n$ and $q_n$ play the same roles as $A_n$, $P_n$ and $Q_n$ respectively in the definition of random continued fractions. In other words, the continued fraction expansion of real numbers is a special model of random continued fractions.

It is worth pointing out that the sequence of partial quotients $\{a_n, n \geq 1\}$ is an ergodic and $\psi$-mixing  stochastic process (see \cite{les84}, Lemma 2.1) and it is not difficult to check that
\[
E(\log a_1) = \int_{[0, 1)}\log a_1(x)d\mu(x) = \dfrac{1}{\log2}\sum\limits_{k= 1}^{\infty}\log k\cdot \log\left(1 + \dfrac{1}{k(k+2)}\right) < \infty.
\]
Therefore, $\{a_n, n \geq 1\}$ satisfies the conditions of Theorem \ref{random constant}. Thus the application of Theorem \ref{random constant} to $\{a_n, n \geq 1\}$ yields that for $\mu$-almost every $x \in [0, 1)$,
\[
\lim\limits_{n \to \infty}\dfrac{1}{n}\log q_n(x) = -\int_{[0, 1)}\log xd\mu = \dfrac{\pi^2}{12\log2},
\]
which is the result of Theorem \ref{real Levy} given by L\'{e}vy \cite{les82} in 1929. Furthermore, for each $0< t< 1$, the expectation
\[
E(a_1^t) = \int_{[0, 1)}a_1^t(x)d\mu(x) = \dfrac{1}{\log2}\sum\limits_{k= 1}^{\infty}k^t\log\left(1 + \dfrac{1}{k(k+2)}\right) < \infty.
\]
Applying Theorem \ref{estimate theorem} to the continued fraction expansion of real numbers, we obtain a Chernoff-type estimate for $\frac{1}{n}\log q_n$ given in Theorem \ref{Levy guji}.

\begin{rem}
(i) Given a distribution of $A_1$, it induces a measure $\nu$ by
$$\nu(B)=P(X\in B),$$
where $B\in\mathcal{B}$. The stationary of $\{A_n, n\geq 1\}$ guarantees that $\nu$ is $T$-invariant, where $T$ is Gauss transformation on $[0, 1)$. Conversely any $T$-invariant measure gives a distribution of $\{a_n, n\geq1\}$ which is a special case of random continued fractions. Therefore, random continued fractions of the form \eqref{random expansion} can completely characterize the set of the invariant measures with respect to $T$.\\
(ii) Theorem \ref{random constant} is parallel to Theorem \ref{real Levy} and Theorem \ref{estimate theorem} is parallel to Theorem \ref{Levy guji};  while Theorem \ref{Levy guji} is even new, to our knowledge.
\end{rem}

At the end, we compare our work with some literature on the large deviation principle for dynamical systems. Orey and Pelikan \cite{lesOP89} provided an explicit description of the rate function of the large deviations for the convergence of trajectory averages of Anosov diffeomorphisms. Young \cite{lesYou90} dealt with the problem of large deviations for continuous maps in compact metric spaces. Kifer \cite{lesKif90} exhibited a unified approach to large deviations of dynamical systems and stochastic processes basing on the existence of a pressure function and the uniqueness of equilibrium states for certain dense sets of functions. Melbourne and Nicol \cite{MN08} studied a class of nonuniformly hyperbolic systems modelled by a Young tower and with a return time function to the base with a exponential or polynomial decay . For more new results, see \cite{lesCT12}, \cite{lesCR11}, \cite{lesKif04}, \cite{Mel09}, \cite{lesRY08} and the references therein. We emphasize that those papers mentioned above do not cover the case of the Gauss transformation since the Gauss transformation is a piecewise map and is not continuous on [0,1). On the other hand, some authors considered the large deviation principle for stationary process under various dependence structures. Orey and Pelikan \cite{lesOP88} established the large deviations for certain classes of stationary processes satisfying a ratio-mixing condition. Bryc \cite{lesBryc92A} proved the large deviation principle for the empirical field of a stationary $\mathbb{Z}^d$-indexed random field under strong mixing dependence assumptions. Bryc \cite{lesBryc92S} also gave the large deviation principle for the arithmetic means of a sequence which has either fast enough $\varphi$-mixing rate or is $\psi$-mixing. Although we assume that the stochastic process $\{A_n, n \geq 1\}$ is $\psi$-mixing, those results about the large deviation estimates for stationary process cannot be applied to obtain our Theorem \ref{estimate theorem} since the mixing property of $\{A_n, n \geq 1\}$ do not transfer to that of the process $\{X_n, n \geq 1\}$ (see Section 2).

\section{Definitions and properties of random continued fractions}
Let $\mathcal{B}(\mathbb{R})$ be the Borel $\sigma$-algebra on $\mathbb{R}$, for any $B_i \in \mathcal{B}(\mathbb{R})$, $1 \leq i \leq n$, $n \in \mathbb{N}$, the set
 \[
\mathcal{C}(B_1 \times B_2\times \cdots \times B_n) = \{x=(x_1,x_2,\cdots): x_i \in B_i, \ \text{for all}\ i = 1,2,\cdots, n\}
\]
is called a cylinder.

Denote
\[
\mathbb{R^{\mathbb{N}}} := \mathbb{R} \times \mathbb{R} \times \cdots = \{x=(x_1,x_2,\cdots): x_i \in \mathbb{R}, \ \text{for all}\ i = 1,2,\cdots\}
\]
and
\[
\mathcal{B}(\mathbb{R}^{\mathbb{N}}):= \mathcal{B}(\mathbb{R}) \times \mathcal{B}(\mathbb{R}) \times \cdots \  \text{is the $\sigma$-algebra generated by all cylinders}.
\]
The measurable space $(\mathbb{R^{\mathbb{N}}},\mathcal{B}(\mathbb{R}^{\mathbb{N}}))$ is called the direct product of the measurable space $(\mathbb{R},\mathcal{B}(\mathbb{R}))$. For more details about the direct product space, we refer to the reader to Shiryaev's book \cite{les78} in Section 2.2.

Let $\{\xi_n, n \geq 1\}$ be a stochastic process defined on the probability space $(\Omega, \mathcal{F}, P)$, taking values in the measurable space $(\mathbb{R},\mathcal{B}(\mathbb{R}))$.

\begin {defn}
A stochastic process $\{\xi_n, n \geq 1\}$ is stationary if for all $n \in \mathbb{N}$,
\[
P((\xi_1, \xi_2, \cdots ) \in B) = P((\xi_{n+1}, \xi_{n+2}, \cdots ) \in B) \ \ \text{for all}\ B \in \mathcal{B}(\mathbb{R}^{\mathbb{N}}).
\]
\end {defn}

\begin {rem}
(i) A sequence of independent and identically distributed (i.i.d.) random variables is stationary.

(ii) Let $(\Omega, \mathcal{F}, P) = (\mathbb{I}, \mathcal{B},\mu)$, where $\mu$ is the Gauss measure. The sequence of partial quotients $\{a_n, n \geq 1\}$ is stationary (see \cite{les79}, \cite{les82}).
\end {rem}

\begin {defn}
A set $A \in \mathcal{F}$ is invariant with respect to the stochastic process $\{\xi_n, n \geq 1\}$ if there is a set $B \in \mathcal{B}(\mathbb{R}^{\mathbb{N}})$ such that for all $n \in \mathbb{N}$,
\[
A = \{\omega: (\xi_{n}(\omega),\xi_{n+1}(\omega),\cdots ) \in B\}.
\]
\end {defn}

\begin {defn}
A stationary stochastic process $\{\xi_n, , n \geq 1\}$ is ergodic if the probability of every invariant set is either 0 or 1.
\end {defn}

\begin {rem}
(i) A sequence of i.i.d. random variables is ergodic.

(ii) The sequence of partial quotients $\{a_n, n \geq 1\}$ is ergodic (see\cite{les79}, \cite{les80}, \cite{les82}, \cite{les83}).
\end {rem}

For all $n \in \mathbb{N}$, let $\sigma(\xi_1,\cdots, \xi_n)$ be the smallest $\sigma$-algebra that makes all $\xi_1, \cdots, \xi_n$ measurable and $\sigma(\xi_1, \cdots, \xi_n, \cdots)$ be the smallest $\sigma$-algebra that makes $\{\xi_n, n \geq 1 \}$ measurable. 

\begin {defn}
A stationary stochastic process $\{\xi_n, , n \geq 1\}$ is $\psi$-mixing if for all $m, n \in \mathbb{N}$ and for any integrable functions $f \in \sigma(\xi_1,\cdots, \xi_m)$, $g \in \sigma(\xi_{m+n}, \xi_{m+n+1}, \cdots)$, we have
\[
\left|E(fg)-E(f)E(g)\right| \leq \psi(n)E(|f|)E(|g|),
\]
where $f \in \sigma(\xi_1,\cdots, \xi_n)$ means that $f$ is a $\sigma(\xi_1,\cdots, \xi_n)$-measurable function, $\psi$ is non-negative with $\psi(n)\to 0$ as $n \to \infty$.
\end {defn}

\begin {rem}
(i) The $\psi$-mixing stochastic process is ergodic.

(ii) The sequence of partial quotients $\{a_n, n \geq 1\}$ is $\psi$-mixing (see \cite{les84}, Lemma 2.1).
\end {rem}

In the following, we will give some properties of random continued fractions. For all $n \in \mathbb{N}$ and any $\omega \in \Omega$, let
\[
X_n(\omega) = \dfrac{1}{A_n(\omega) +\dfrac{1}{A_{n+1}(\omega) + \ddots}},
\]
it is clear that for any $\omega \in \Omega$, $X_1(\omega) = X(\omega)$ by form (\ref{random expansion}). 

\begin {pr}
 With the convention, $P_{-1} \equiv 1$, $Q_{-1} \equiv 0$, $P_0 \equiv 0$, $Q_0 \equiv 1$. The following properties hold for all $n \in \mathbb{N}$ and any $\omega \in \Omega$.
\begin{align*}
&(i)\ P_n = A_nP_{n-1}+ P_{n-2}, \ Q_n= A_nQ_{n-1}+ Q_{n-2}.& \\
&Moreover,\ Q_n \geq Q_{n-1}+ Q_{n-2} \ \ and\ \  Q_n \geq 2^{\frac{n-1}{2}}.&\\
&(ii)\ P_nQ_{n-1}- Q_nP_{n-1}=(-1)^{n-1}.&\\
&(iii)\ P_nQ_{n-2}- Q_nP_{n-2}=(-1)^nA_n.&\\
&(iv)\ X_{n+1} = -\dfrac{X_1Q_n-P_n}{X_1Q_{n-1}-P_{n-1}}.\ That\ is,\ X_1 = \dfrac{P_n + X_{n+1}P_{n-1}}{Q_n + X_{n+1}Q_{n-1}}.&\\
&(v)\ \dfrac{1}{2Q_nQ_{n-1}} \leq \left|X_1-\dfrac{P_{n-1}}{Q_{n-1}}\right| \leq \dfrac{1}{Q_nQ_{n-1}}.&\\
&(vi)\ X_1\cdot X_2\cdot \cdots \cdot X_n= \left|X_1Q_{n-1}- P_{n-1}\right|.&
\end{align*}
\end {pr}

\begin{rem}
All proofs of the above properties are similar to that of the continued fraction expansion of real numbers (see \cite{les79}, \cite{les80}, \cite{les81}, \cite{les76}).
\end{rem}

\begin{pr}
For all $n \geq 1$ and any $\omega \in \Omega$, we have
\[
\left|\dfrac{1}{n}\log Q_n + \dfrac{1}{n}\sum_{k=1}^n\log X_k \right| \leq \dfrac{1}{n}\log 2.
\]
\end{pr}

\begin{proof}

By Proposition 1 (v) and (vi), we have
\[
\dfrac{1}{2Q_n} \leq X_1\cdot X_2\cdot \cdots \cdot X_n \leq \dfrac{1}{Q_n}.
\]

Taking the logarithm on both sides of the above inequalities and divided by $n$, we obtain that
\[
-\dfrac{1}{n}\log Q_n - \dfrac{1}{n}\log 2 \leq \dfrac{1}{n}\sum_{k=1}^n\log X_k  \leq -\dfrac{1}{n}\log Q_n.
\]

Therefore,
\begin{align*}
\left|\dfrac{1}{n}\log Q_n + \dfrac{1}{n}\sum_{k=1}^n\log X_k \right| \leq \dfrac{1}{n}\log 2.
\end{align*}
\end{proof}

\begin{rem}
By Proposition 2, we know that for any $\varepsilon > 0$, there exists $N > 0$ such that for all $\omega \in \Omega$ and all $n \geq N$, we have
\[
\left|\dfrac{1}{n}\log Q_n(\omega) + \dfrac{1}{n}\sum_{k=1}^n\log X_k(\omega) \right| < \varepsilon.
\]
This means that $\lim\limits_{n \to \infty}\left(\dfrac{1}{n}\log Q_n + \dfrac{1}{n}\sum_{k=1}^n\log X_k \right) =0$ uniformly.
\end{rem}

Next we investigate which properties of $X$ can be inherited from $A$.

\begin{pr}\label{h}
Let $h: (\{1, 2, \cdots, n, \cdots\}^{\mathbb{N}}, \mathcal{S})) \longrightarrow  \left((0, 1), \mathcal{B}((0, 1))\right)$ defined by
\[
h(x_1, x_2, \cdots, x_n, \cdots) = \dfrac{1}{x_1 +\dfrac{1}{x_2 + \ddots +\dfrac{1}{x_n + \ddots}}},
\]
where the measurable space $(\{1, 2, \cdots, n, \cdots\}^{\mathbb{N}}, \mathcal{S})$ is the direct product of measurable space $(\{1, 2, \cdots, n, \cdots\}, \mathcal{C})$. Then $h$ is a measurable function.
\end{pr}

\begin{proof}
Let
\[
\mathcal{A}= \{I(a_1, a_2, \cdots, a_n): a_i \in \{1, 2, \cdots\},\ \text{for all} \ n \in \mathbb{N},\ i = 1,\cdots,n\},
\]
where $I(a_1, a_2, \cdots, a_n) = \{x \in (0, 1): a_1(x)=a_1, a_2(x)=a_2, \cdots, a_n(x)=a_n\}$, $a_i \in \{1, 2, \cdots \}$  for each $1 \leq i \leq n$, $n \in \mathbb{N}$.

It is well-known that $\mathcal{A}$ is a semi-algebra and can generate the Borel $\sigma$-algebra $\mathcal{B}((0, 1))$. So it suffices to show that
for every $I(a_1, a_2, \cdots, a_n) \in \mathcal{A}$, we have
$$h^{-1}(I(a_1, a_2, \cdots, a_n)) \in \mathcal{S}.$$
In fact,
\begin{align*}
h^{-1}(I(a_1, a_2, \cdots, a_n))= &\{(x_1, \cdots, x_n, \cdots): h(x_1, \cdots, x_n, \cdots) \in I(a_1, \cdots, a_n)\}\\
=&\{(x_1, \cdots, x_n, \cdots): x_1= a_1, \cdots, x_n= a_n\}\\
=&\{a_1\}\times\cdots\times\{a_n\}\times\mathbb{N}\times \cdots.
\end{align*}
Therefore,
\[
h^{-1}(I(a_1, a_2, \cdots, a_n)) \in \mathcal{S}.
\]
\end{proof}

\begin {pr}[\cite{les75}, Theorem 6.1.1 and Theorem 6.1.3]
 Let $\{\xi_n, n \geq 1\}$ be a stochastic process, $F: \mathbb{R}^{\mathbb{N}} \longrightarrow \mathbb{R}$ be a measurable function and $\eta_n = F(\xi_n, \xi_{n+1}, \cdots)$ for all $n \geq 1$. Then if $\{\xi_n, n \geq 1\}$ is stationary, so is $\{\eta_n, n \geq 1\}$; and if $\{\xi_n, n \geq 1\}$ is ergodic, so is $\{\eta_n, n \geq 1\}$.
\end {pr}

By Proposition 3, Proposition 4 and $X_n = h(A_n, A_{n+1}, \cdots)$, we can immediately obtain the following corollary.
\begin{cor}
If the stochastic process $\{A_n, n \geq 1\}$ is stationary and ergodic, then $\{X_n, n \geq 1\}$ is also stationary and ergodic respectively.
\end{cor}

\section{Proofs of Theorem \ref{random constant} and Theorem \ref{estimate theorem}}

The following ergodic theorem (see \cite{les78}, Section 5.3) will be needed in the proof of Theorem \ref{random constant}.

\begin{thm}[Ergodic theorem \cite{les78}]\label{theorem 2}
Let $\{\xi_n, n \geq 1\}$ be an ergodic stochastic process. Then for every real-valued measurable function $f$ with $E(|f(\xi_1)|)< \infty$, we have
\begin{equation*}
\lim\limits_{n \to \infty}\frac{1}{n}\sum_{k = 1}^nf(\xi_k(\omega)) = E\left(f(\xi_1)\right),\ \ \ \ P-a.s.
\end{equation*}
\end{thm}

Now, we are ready to prove Theorem \ref{random constant}.
\begin{proof}[Proof of Theorem \ref{random constant}]
Since $\{A_n, n\geq 1\}$ is an ergodic stochastic process, by Corollary 1, we obtain that $\{X_n, n\geq 1\}$ is ergodic.

Note that $(A_1+1)^{-1} < X_1 \leq A_1^{-1}$ and
$E\left(\log A_1\right) < \infty$, so $E\left(\left|\log X_1\right|\right) = -E\left(\log X_1\right) \leq E\left(\log(A_1+1)\right) \leq E\left(\log2A_1\right) =\log2 + E(\log A_1) < \infty$. By Theorem 3, we have
\begin{align*}
\lim\limits_{n \to \infty}\dfrac{1}{n}\sum_{k=1}^n\log X_k  = \int_{\Omega}\log X dP,\ \ \ \ P-a.s.
\end{align*}

On the other hand, by Proposition 2, for any $\varepsilon > 0$, there exists $N > 0$ such that for all $\omega \in \Omega$ and all $n \geq N$, we have
\[
\left|\dfrac{1}{n}\log Q_n(\omega) + \dfrac{1}{n}\sum_{k=1}^n\log X_k(\omega) \right| < \varepsilon.
\]
Therefore,
\[
\lim\limits_{n \to \infty}\dfrac{1}{n}\log Q_n = -\int_{\Omega}\log X dP,\ \ \ \  P-a.s.
\]
\end{proof}

The following Lemma 1, Lemma 2 and Lemma 3 are the key lemmas for proving Theorem \ref{estimate theorem}.

\begin {lem}
For any $\delta > 0$, there exists $N_0 > 0$, such that for all $ n > N_0$, we have
\[
\left\{\omega: \left|\dfrac{1}{n}\log Q_n(\omega) + \int_{\Omega}\log XdP\right| \geq \delta\right\}
\subseteq \left\{\omega: \left|\dfrac{1}{n}\sum_{k=1}^n\log X_k(\omega) - \int_{\Omega}\log XdP\right| \geq \frac{\delta}{2}\right\}.
\]
\end {lem}

\begin{proof}
It suffices to prove that if for any $\delta > 0$, there exists $N_0 > 0$,
\begin{align}\label{tiaojian dayu}
\left|\dfrac{1}{n}\log Q_n + \int_{\Omega}\log X dP\right| \geq \delta
\end{align}
holds for all $n > N_0$, then
\[
\left|\dfrac{1}{n}\sum_{k=1}^n\log X_k - \int_{\Omega}\log X dP\right| \geq \frac{\delta}{2}.
\]

Since
\begin{align*}
\lim\limits_{n \to \infty} \dfrac{1}{n}\log 2= 0,
\end{align*}
by Proposition 2, we obtain that for any $\delta > 0$, there exists $N_0 >0$ such that for all $\ n > N_0$, we have
\begin{align}\label{yizhi xiaoyu}
\left|\dfrac{1}{n}\log Q_n + \dfrac{1}{n}\sum_{k=1}^n\log X_k \right| \leq \dfrac{1}{n}\log 2 < \frac{\delta}{2}.
\end{align}
By the triangle inequality, we deduce that for all $n > N_0$,
\begin{align*}
& \left|\dfrac{1}{n}\sum_{k=1}^n\log X_k - \int_{\Omega}\log X dP\right|\\
=& \left|\dfrac{1}{n}\sum_{k=1}^n\log X_k +\dfrac{1}{n}\log Q_n - \dfrac{1}{n}\log Q_n - \int_{\Omega}\log X dP\right|\\
\geq & \left|\dfrac{1}{n}\log Q_n + \int_{\Omega}\log X dP\right| - \left|\dfrac{1}{n}\log Q_n + \dfrac{1}{n}\sum_{k=1}^n\log X_k\right| \geq \frac{\delta}{2},
\end{align*}
where the last inequality follows from (\ref{tiaojian dayu}) and (\ref{yizhi xiaoyu}). This completes the proof.
\end{proof}

\begin {lem}
Let $\{A_n, n \geq 1\}$ be a $\psi$-mixing stochastic process and $0< t< 1$. Then for any $\varepsilon > 0$, there exists $N > 0$ such that for any integrable function $g \in \sigma(A_N, A_{N+1}, \cdots)$, we have
\[
E\left(X_1^t|g|\right) \leq (1 + \varepsilon)^2E\left(X_1^t\right)E\left(|g|\right).
\]
\end {lem}

\begin{proof}
Recall that, since $\{A_n, n \geq 1\}$ is a $\psi$-mixing stochastic process, for all $m, n \in \mathbb{N}$ and for any integrable functions $f \in \sigma(A_1,\cdots, A_m)$, $g \in \sigma(A_{m+n}, A_{m+n+1}, \cdots)$, we have
\[
\left|E(fg)-E(f)E(g)\right| \leq \psi(n)E\left(|f|\right)E\left(|g|\right),
\]
with $\psi(n)\to 0$ as $n \to \infty$ being $0 < E(X_1^t) \leq 1$.

For any $\varepsilon > 0$, we can choose $N_1 > 0$, such that $\psi(N_1) \leq \varepsilon$ and

\begin{align}\label{xiaoyu qiwang}
2^{-(N_1 -1)t} \leq \dfrac{1}{2}\varepsilon E\left(X_1^t\right)
\end{align}
holds. The first implies
\begin{align}\label{mixing inequality}
\left|E(fg)-E(f)E(g)\right| \leq \varepsilon E\left(|f|\right)E\left(|g|\right).
\end{align}

For all $m \in \mathbb{N}$, let $Y_m= \frac{P_m}{Q_m}$. Then  $Y_m \in \sigma(A_1,\cdots, A_m)$. Proposition 1 (i) and (v) imply that
\begin{align}\label{Lemma2 in1}
\left|X_1 - Y_m \right| \leq \dfrac{1}{Q_mQ_{m+1}} \leq \dfrac{1}{Q_m^2} \leq 2^{-(m-1)}.
\end{align}

Note that $X_1 > 0$, $Y_{N_1} >0$, $0< t <1$, we know that $\left|X_1^t - Y_{N_1}^t\right|\leq \left|X_1 - Y_{N_1}\right|^t$. Applying $m = N_1$ in (\ref{Lemma2 in1}) leads to $\left|X_1^t - Y_{N_1}^t\right|\leq 2^{-(N_1- 1)t}$, that is,
\begin{align}\label{liangbian}
 -2^{-(N_1- 1)t} \leq X_1^t - Y_{N_1}^t \leq 2^{-(N_1- 1)t}.
\end{align}

Therefore,
\begin{align}\label{jiben1}
E\left(X_1^t|g|\right) \leq E\left((Y_{N_1}^t + 2^{-(N_1 -1)t})|g|\right) =E\left(Y_{N_1}^t|g|\right) + 2^{-(N_1 -1)t}E\left(|g|\right),
\end{align}
for any integrable function $g \in \sigma(A_N, A_{N+1}, \cdots)$ with $N = 2N_1$.

Since $Y_{N_1} \in \sigma(A_1,\cdots, A_{N_1})$, $g \in \sigma(A_N, A_{N+1}, \cdots)$ with $N = 2N_1$, by (\ref{xiaoyu qiwang}) and (\ref{mixing inequality}), we deduce that
\begin{align}\label{jiben2}
E\left(Y_{N_1}^t|g|\right) + 2^{-(N_1 -1)t}E\left(|g|\right) \leq E\left(Y_{N_1}^t\right)E(|g|)(1 + \varepsilon) + \dfrac{1}{2}\varepsilon E\left(X_1^t\right)E\left(|g|\right).
\end{align}
The first inequality of (\ref{liangbian}) implies that
\begin{align}\label{jiben3}
E\left(Y_{N_1}^t\right) \leq E\left(X_1^t +2^{-(N_1 -1)t}\right) = E\left(X_1^t\right) + 2^{-(N_1 -1)t} \leq
(1 + \dfrac{1}{2}\varepsilon)E\left(X_1^t\right),
\end{align}
where the last inequality follows from (\ref{xiaoyu qiwang}).

Combining the inequalities (\ref{jiben1}), (\ref{jiben2}) and (\ref{jiben3}), for any $\varepsilon > 0$ and for any integrable function $g \in \sigma(A_N, A_{N+1}, \cdots)$, we have
\begin{align*}
E\left(X_1^t|g|\right) & \leq (1 + \dfrac{1}{2}\varepsilon)(1 + \varepsilon)E\left(X_1^t\right)E\left(|g|\right) + \dfrac{1}{2}\varepsilon E\left(X_1^t\right)E\left(|g|\right)\\
& \leq (1 + \varepsilon)^2E\left(X_1^t\right)E\left(|g|\right).
\end{align*}
\end{proof}

\begin {lem}
Suppose that for each $0< t< 1$, $E(A_1^t) < \infty$. Let $\{A_n, n \geq 1\}$ be a $\psi$-mixing stochastic process. Then for any $0< t< 1$ and for any $\varepsilon > 0$, there exists $N > 0$ such that for any integrable function $g \in \sigma(A_N, A_{N+1}, \cdots)$, we have
\[
E\left(X_1^{-t}|g|\right) \leq (1 + \varepsilon)^2E\left(X_1^{-t}\right)E\left(|g|\right).
\]
\end {lem}

\begin{proof}
Since $(A_1+1)^{-1} < X_1 \leq A_1^{-1}$ and $E(A_1^t)< \infty$ for each $0< t< 1$, we have
\[
0 < E\left(A_1^t\right) \leq E\left(X_1^{-t}\right) \leq E\left((A_1 + 1)^t\right)\leq 2^tE\left(A_1^t\right) < \infty.
\]
Since $\{A_n, n \geq 1\}$ is a $\psi$-mixing stochastic process, for any $\varepsilon > 0$,
we can choose $N_2 > 0$, such that for all $m \in \mathbb{N}$ and for any integrable function $f \in \sigma(A_1,\cdots, A_m)$, $g \in \sigma(A_{m+N_2}, A_{m+N_2+1}, \cdots)$, we obtain that
\begin{align}\label{mixing 2}
\left|E\left(fg\right)-E\left(f\right)E\left(g\right)\right| \leq \varepsilon E\left(|f|\right)E\left(|g|\right)
\end{align}
and
\begin{align}\label{xiaoyu qiwang2}
2^{-(N_2 -2)t} \leq \dfrac{1}{2}\varepsilon E\left(X_1^{-t}\right).
\end{align}

For all $m \in \mathbb{N}$, let $Y_m= \frac{P_m}{Q_m}$, so $Y_m \in \sigma(A_1,\cdots, A_m)$. Note that $X_1^{-1} = A_1 + X_2$ and $Y_m^{-1} = A_1 + Y_m^{'}$, where $Y_m^{'}= [A_2,A_3,\cdots, A_m ]$.

Therefore,
\begin{align}\label{Lemma3 in1}
\left|X_1^{-1} - Y_m^{-1}\right|= \left|X_2 - Y_m^{'}\right| \leq 2^{-(m-2)},
\end{align}
where the last inequality follows from (\ref{Lemma2 in1}).

Notice that $X_1^{-1} > 0$, $Y_m^{-1} >0$ and $0< t <1$, we know that $\left|X_1^{-t} - Y_m^{-t}\right|\leq \left|X_1^{-1} - Y_m^{-1}\right|^t$. Applying $m = N_2$ in (\ref{Lemma3 in1}) leads to $\left|X_1^{-t} - Y_{N_2}^{-t}\right|\leq 2^{-(N_2- 2)t}$, that is,
\begin{align}\label{liangbian2}
-2^{-(N_2- 2)t} \leq X_1^{-t} - Y_{N_2}^{-t} \leq 2^{-(N_2- 2)t}.
\end{align}

Therefore,
\begin{align}\label{L3jiben1}
E\left(X_1^{-t}|g|\right) \leq E\left((Y_{N_2}^{-t} + 2^{-(N_2 -2)t})|g|\right) = E\left(Y_{N_2}^{-t}|g|\right) + 2^{-(N_2 -2)t}E\left(|g|\right),
\end{align}
for any integrable function $g \in \sigma(A_N, A_{N+1}, \cdots)$ with $N = 2N_2$.

Since $Y_{N_2}^{-1} \in \sigma(A_1,\cdots, A_{N_2})$ and $g \in \sigma(A_N, A_{N+1}, \cdots)$ with $N = 2N_2$, by (\ref{mixing 2}) and (\ref{xiaoyu qiwang2}), we deduce that
\begin{align}\label{L3jiben2}
E\left(Y_{N_2}^{-t}|g|\right) + 2^{-(N_2 -2)t}E\left(|g|\right) \leq E\left(Y_{N_2}^{-t}\right)E\left(|g|\right)(1 + \varepsilon) + \dfrac{1}{2}\varepsilon E\left(X_1^{-t}\right)E\left(|g|\right).
\end{align}
The first inequality of (\ref{liangbian2}) implies that
\begin{align}\label{L3jiben3}
E\left(Y_{N_2}^{-t}\right) \leq E\left(X_1^{-t} +2^{-(N_2 -2)t}\right) = E\left(X_1^{-t}\right) + 2^{-(N_2 -2)t} \leq
(1 + \dfrac{1}{2}\varepsilon)E\left(X_1^{-t}\right),
\end{align}
where the last inequality follows from (\ref{xiaoyu qiwang2}).

In conjunction with the inequalities (\ref{L3jiben1}), (\ref{L3jiben2}) and (\ref{L3jiben3}), for any $\varepsilon > 0$ and for any integrable function $g \in \sigma(A_N, A_{N+1}, \cdots)$, we have
\begin{align*}
E\left(X_1^{-t}|g|\right) & \leq (1 + \dfrac{1}{2}\varepsilon)(1 + \varepsilon)E\left(X_1^{-t}\right)E\left(|g|\right) + \dfrac{1}{2}\varepsilon E\left(X_1^{-t}\right)E\left(|g|\right)\\
& \leq (1 + \varepsilon)^2E\left(X_1^{-t}\right)E\left(|g|\right).
\end{align*}
\end{proof}

Now we are ready to prove Theorem \ref{estimate theorem}. The main idea is from the technique of Chernoff-type estimate for the i.i.d. sequence (see \cite{les75}, Section 1.9). The difficulty here is that we do not have the independence assumption.

\begin{proof}[Proof of Theorem \ref{estimate theorem}]
For any $\delta > 0$, by Lemma 1, there exists $N_0 > 0$, such that for all $n > N_0$, we have
\begin{align*}
&P\left(\left|\dfrac{1}{n}\log Q_n + \int_{\Omega}\log X dP\right| \geq \delta\right)
\leq P\left(\left|\dfrac{1}{n}\sum_{k=1}^n\log X_k - \int_{\Omega}\log X dP\right| \geq \frac{\delta}{2}\right),\\
\end{align*}
which indicates that it is enough to estimate
\[
U:= P\left(\left|\dfrac{1}{n}\sum_{k=1}^n\log X_k - \int_{\Omega}\log X dP\right| \geq \frac{\delta}{2}\right)
\]
from above.

Obviously, $U$ can be divided into the following two parts
\[
\uppercase\expandafter{\romannumeral1} := P\left(\sum_{k=1}^n\log X_k  \geq  n(\int_{\Omega}\log X dP + \frac{\delta}{2})\right)
\]
and
\[
\uppercase\expandafter{\romannumeral2} := P\left(\sum_{k=1}^n\log X_k  \leq  n(\int_{\Omega}\log X dP - \frac{\delta}{2})\right),
\]
since $U = \uppercase\expandafter{\romannumeral1} + \uppercase\expandafter{\romannumeral2}$.

In the following, we will estimate $\uppercase\expandafter{\romannumeral1}$ and $\uppercase\expandafter{\romannumeral2}$ respectively.

First, we estimate $\uppercase\expandafter{\romannumeral1}$.
Let $\lambda > 0$ be a parameter, the quantity $\uppercase\expandafter{\romannumeral1}$ can be written as
\[
P\left(\exp({\lambda\sum_{k=1}^n\log X_k}) \geq \exp({n\lambda(\int_{\Omega}\log X dP + \frac{\delta}{2})})\right)
\]
and the Markov's inequality deduces that
\begin{align}\label{Markov}
\uppercase\expandafter{\romannumeral1} \leq & \exp\left({-n\lambda(\int_{\Omega}\log X dP + \frac{\delta}{2})}\right)E\left(\exp(\lambda\sum_{k=1}^n\log X_k)\right)\\
= & \exp\left({-n\lambda(\int_{\Omega}\log X dP + \frac{\delta}{2})}\right)E\left(X_1^{\lambda}X_2^{\lambda} \cdots X_n^{\lambda}\right).\notag
\end{align}

Now, it remains to estimate $E\left(X_1^{\lambda}X_2^{\lambda} \cdots X_n^{\lambda}\right)$. If $X_1, X_2, \cdots, X_n$ are independent for all $n \geq 1$, then
\begin{align}\label{duli}
E\left(X_1^{\lambda}X_2^{\lambda} \cdots X_n^{\lambda}\right) = \prod\limits_{i=1}^nE\left(X_i^{\lambda}\right).
\end{align}
However, here $X_1, X_2, \cdots, X_n$ are not independent, since $\{A_n, n \geq 1\}$ are not independent. It means that we do not have the equality (\ref{duli}).
We can improve the techniques about Chernoff-type estimate for the independent random variables (see \cite{les75}, Section 1.9) by Lemma 2.


Notice that $|e^z-1| \leq e^{|z|}-1 \leq |z|e^{|z|}$ for every $z \in \mathbb{R}$, so for any $0< \epsilon < 1$, $|\theta| \leq \epsilon $ and $0 < x < 1$, we have
\begin{equation}\label{budengshi}
\left|x^\theta -1\right| = |e^{\theta \log x} -1| \leq |\theta\log x|e^{|\theta \log x|} \leq |\theta\log x|
\left(\frac{1}{x}\right)^{\epsilon}.
\end{equation}
For any $0 < t< 1$ and $0< \varepsilon< t$, let $f_{\theta}(x) = (x^{t+\theta}-x^{t})/\theta$ and $\varphi(x) = \left(\frac{1}{x}\right)^{\varepsilon}\log \frac{1}{x}$ for every $x \in (0,1)$, where $0<|\theta| \leq \epsilon$. It is clear that $|f_{\theta}(x)| \leq \varphi(x)$ and $f_{\theta}(x) \rightarrow x^t\log x$ as $\theta \to 0$. Since $(A_1+1)^{-1} < X_1 \leq A_1^{-1} \leq 1$ and our assumption $E(A_1^t)< \infty$ for each $0< t< 1$, we have $E(\varphi(X_1)) \leq E\left((A_1+1)^\varepsilon\log(A_1+1)\right) \leq 2^\varepsilon E(A_1^\varepsilon\log2A_1) < \infty$.

Let $F(x) = P(X_1 \leq x)$. An application of the dominated convergence theorem shows that the derivative
\begin{align}\label{derivate}
\left(E(X_1^t)\right)^{'} &= \lim\limits_{\theta \to 0}\frac{E(X_1^{t+\theta}) - E(X_1^t)}{\theta} \\ \notag
& = \lim\limits_{\theta \to 0} \int \dfrac{x^\theta - 1}{\theta}x^t dF(x) \\ \notag
& = \int x^t\log x dF(x) = E\left(X_1^t\log X_1\right)\notag.
\end{align}
Moreover, $E(|X_1^t\log X_1|) \leq E(-\log X_1) \leq E(\log (A_1+1)) \leq \log 2 +  E(\log A_1) < \infty$ since $E(A_1^t)< \infty$ for each $0< t< 1$.
By L'Hospital's rule and (\ref{derivate}), we deduce that
\begin{align*}
\lim\limits_{t \to 0^+}\dfrac{\log E\left(X_1^t\right)}{t} = \lim\limits_{t \to 0^+}\dfrac{E\left(X_1^t\log X_1\right)}{E\left(X_1^t\right)} = \int_{\Omega}\log X dP,
\end{align*}
where the second equality follows from the dominated convergence theorem and our assumption $E(A_1^t)< \infty$ for each $0< t< 1$. Therefore, we can choose $ 0< t_0 < 1$, such that
\begin{align}\label{zuihou 1}
\left|\dfrac{\log E\left(X_1^{t_0}\right)}{t_0} - \int_{\Omega}\log X dP\right| < \frac{\delta}{8}.
\end{align}

Let $\varepsilon = \dfrac{1}{8}\delta t_0$.
By Lemma 2, there exists $N_1 > 0$, such that for any integrable function $g \in \sigma(A_{N_1}, A_{N_1 + 1}, \cdots)$, we have
\begin{align}\label{zuihou 2}
E\left(X_1^{t_0}|g|\right) \leq (1 + \varepsilon)^2E\left(X_1^{t_0}\right)E\left(|g|\right).
\end{align}

 Choose $\lambda = \dfrac{t_0}{N_1}$. The following proof will be divided into two cases according to whether $n$ can be divided by $N_1$.

 Case 1. $n = kN_1$ for some $k \in \mathbb{N}$. Denote
\[
\Lambda_i = \{i, N_1 +i, \cdots, (k-1)N_1 +i \} \ (\ 1 \leq i \leq N_1),
\]
then $\Lambda_i \cap \Lambda_j = \emptyset$ for any $1 \leq i \neq j \leq N_1$ and $\bigcup\limits_{i = 1}^{N_1} \Lambda_i = \{1, 2, \cdots, n\}.$

Write
\begin{align}\label{chengji qiwang}
E\left(X_1^\lambda X_2^\lambda \cdots X_n^\lambda\right) =  E\left(\prod_{i=1}^{N_1} \prod_{j \in \Lambda_i}X_j^\lambda\right).
\end{align}

By H\"{o}lder's inequality, we obtain that
\begin{align}\label{Holder}
E\left(\prod_{i=1}^{N_1} \prod_{j \in \Lambda_i}X_j^\lambda\right)
\leq \  \prod_{i=1}^{N_1} \left(E\left(\prod_{j \in \Lambda_i}X_j^{\lambda N_1}\right)\right)^{\frac{1}{N_1}}
= \left(\prod_{i=1}^{N_1} E\left(\prod_{j \in \Lambda_i}X_j^{t_0}\right)\right)^{\frac{1}{N_1}}.
\end{align}

Applying the inequality (\ref{zuihou 2}) for $g = X_{N_1 + 1}^{t_0} \cdots X_{(k-1)N_1 + 1}^{t_0} \in \sigma(A_{N_1 + 1}, A_{N_1 + 2}, \cdots)$, we deduce that
\begin{align*}
E\left(X_1^{t_0}X_{N_1 + 1}^{t_0} \cdots X_{(k-1)N_1 + 1}^{t_0}\right) \leq &\  E\left(X_1^{t_0}\right)E\left(X_{N_1 + 1}^{t_0} \cdots X_{(k-1)N_1 + 1}^{t_0}\right)(1 + \varepsilon)^2.
\end{align*}
Notice that
\begin{align*}
E\left(X_{N_1 + 1}^{t_0} X_{2N_1 + 1}^{t_0}\cdots X_{(k-1)N_1 + 1}^{t_0}\right)= E\left(X_ 1^{t_0} X_{N_1 + 1}^{t_0}\cdots X_{(k-2)N_1 + 1}^{t_0}\right),
\end{align*}
since $\{X_n, n\geq 1\}$ is stationary. Therefore, repeating this procedure $k - 1$ steps, we can obtain the following inequality
\begin{align}\label{zuihou 3}
E\left(\prod_{j \in \Lambda_1}X_j^{t_0}\right) 
\leq \left(E(X_1^{t_0})\right)^k(1 + \varepsilon)^{2(k-1)}.
\end{align}

Similarly, since $\{X_n, n\geq 1\}$ is stationary, we deduce that
\begin{align}\label{zuihou 4}
E\left(\prod_{j \in \Lambda_i}X_j^{t_0}\right) \leq \left(E\left(X_1^{t_0}\right)\right)^k(1 + \varepsilon)^{2(k-1)} \ (i = 2, 3, \cdots, N_1).
\end{align}

Combining (\ref{chengji qiwang}), (\ref{Holder}), (\ref{zuihou 3}) and (\ref{zuihou 4}), we have
\begin{align*}
&E\left(X_1^\lambda X_2^\lambda \cdots X_n^\lambda\right) \leq \left(\prod_{i=1}^{N_1}\left(E\left(X_1^{t_0}\right)\right)^k(1 + \varepsilon)^{2(k-1)}\right)^{\frac{1}{N_1}} \\
& = \left(\left(E(X_1^{t_0})\right)^n (1 + \varepsilon)^{2(n - N_1)}\right)^{\frac{1}{N_1}}
\leq \left(\left(E(X_1^{t_0})\right)^n (1 + \varepsilon)^{2n} \right)^{\frac{1}{N_1}}.
\end{align*}
Therefore, write $\left((E\left(X_1^{t_0}\right))^n (1 + \varepsilon)^{2n} \right)^{\frac{1}{N_1}}$ as the exponential form, we obtain that 
\begin{align*}
\uppercase\expandafter{\romannumeral1} \leq & \exp\left({-n\lambda(\int_{\Omega}\log X dP + \frac{\delta}{2})}\right)E\left(X_1^{\lambda}X_2^{\lambda} \cdots X_n^{\lambda}\right)\\
\leq & \exp\left({-n\lambda(\int_{\Omega}\log X dP + \frac{\delta}{2} - \frac{\log E(X_1^{t_0})}{t_0} - \frac{2}{t_0}\log (1 + \varepsilon))}\right).
\end{align*}

Since $\log(1 + x) \leq x$ for all $x \geq 0$, we know that
\begin{align}\label{log inequality}
\dfrac{2}{t_0}\log (1 + \varepsilon) \leq \dfrac{2}{t_0}\varepsilon = \dfrac{\delta}{4}.
\end{align}

Together with (\ref{zuihou 1}) and (\ref{log inequality}), we deduce that
\begin{align*}
\frac{\delta}{2} + \int_{\Omega}\log X dP - \dfrac{\log E(X_1^{t_0})}{t_0} - \dfrac{2}{t_0}\log (1 + \varepsilon)
\geq \frac{\delta}{2} - \frac{\delta}{8} - \frac{\delta}{4} = \frac{\delta}{8}.
\end{align*}

Therefore, 
\begin{align*}
\uppercase\expandafter{\romannumeral1} \leq \exp\left({-\frac{1}{8}\lambda\delta n}\right).
\end{align*}

Case 2. $n = kN_1 + l$ for some $k \in \mathbb{N}$ and $0 < l < N_1$. Since $X_i^\lambda \leq 1$ for $kN_1 + 1 \leq i \leq n$, we have
$$E\left(X_1^{\lambda}X_2^{\lambda} \cdots X_n^{\lambda}\right) \leq E\left(X_1^{\lambda}X_2^{\lambda} \cdots X_{kN_1}^{\lambda}\right).$$
Thus, by (\ref{Markov}), we deduce that
\begin{align*}
\uppercase\expandafter{\romannumeral1} \leq \exp\left({-kN_1\lambda(\int_{\Omega}\log X(\omega)dP+ \frac{\delta}{2})}\right)E\left(X_1^{\lambda}X_2^{\lambda} \cdots X_{kN_1}^{\lambda}\right) \exp\left({-l\lambda(\int_{\Omega}\log X(\omega)dP + \frac{\delta}{2})}\right).
\end{align*}

By the result in Case 1 for $kN_1$, we obtain that
\begin{align}\label{case2 in1}
\uppercase\expandafter{\romannumeral1} \leq & \exp({-\frac{1}{8}\lambda\delta kN_1})\exp\left({-l\lambda(\int_{\Omega}\log X dP + \frac{\delta}{2})}\right)\\
= &\exp({-\frac{1}{8}\lambda\delta n}) \exp\left(-l\lambda(\int_{\Omega}\log X dP + \frac{3}{8}\delta)\right)\notag.
\end{align}

Let $B_1 = \max\left\{1, \exp\left(-t_0(\int_{\Omega}\log XdP + \dfrac{3}{8}\delta)\right)\right\}$ and $\alpha_1 = \dfrac{1}{8}\lambda\delta$.
So by (\ref{case2 in1}), for any $\delta > 0$ and all $ n \geq N_1$, we have
\[
\uppercase\expandafter{\romannumeral1} \leq B_1\exp(-\alpha_1 n).
\]

Next, we estimate $\uppercase\expandafter{\romannumeral2}$. Let $\tau > 0$ be a parameter, the quantity $\uppercase\expandafter{\romannumeral2}$ can be written as
\[
P\left(\exp({-\tau\sum_{k=1}^n\log X_k}) \geq \exp({-n\tau(\int_{\Omega}\log X dP - \frac{\delta}{2})})\right)
\]
and the Markov's inequality deduces that
\begin{align}\label{fu Markov}
\uppercase\expandafter{\romannumeral2} \leq & \exp\left({n\tau(\int_{\Omega}\log X dP - \frac{\delta}{2})}\right)E\left(\exp(-\tau\sum_{k=1}^n\log X_k)\right)\\
= & \exp\left({n\tau(\int_{\Omega}\log X dP - \frac{\delta}{2})}\right)E\left(X_1^{-\tau} X_2^{-\tau} \cdots X_n^{-\tau}\right).\notag
\end{align}

It remains to estimate $E(X_1^{-\tau}X_2^{-\tau} \cdots X_n^{-\tau})$. In view of (\ref{budengshi}), for any $0< \epsilon < 1$, $|\theta| \leq \epsilon $ and $0 < x < 1$, we have
\begin{equation*}
\left|x^{-\theta} -1\right| \leq |\theta\log x|\left(\frac{1}{x}\right)^{\epsilon}.
\end{equation*}
For any $0 < t< 1$ and $0< \varepsilon< 1-t$, let $f_{\theta}(x) = (x^{-(t+\theta)}-x^{-t})/\theta$ and $\varphi(x) = \left(\frac{1}{x}\right)^{t+\varepsilon}\log \frac{1}{x}$ for every $x \in (0,1)$, where $0<|\theta| \leq \epsilon$. It is clear that $|f_{\theta}(x)| \leq \varphi(x)$ and $f_{\theta}(x) \rightarrow -x^{-t}\log x$ as $\theta \to 0$. Since $(A_1+1)^{-1} < X_1 \leq A_1^{-1} \leq 1$ and $E(A_1^t)< \infty$ for each $0< t< 1$, we have $E(\varphi(X_1)) < \infty$. Being similar to (\ref{derivate}), we have the derivative $(E(X_1^{-s}))^{\prime} = -E(X_1^{-s}\log X_1)$ for any $0 < s <1$. Moreover, $E(|X_1^{-s}\log X_1|) \leq E\left((A_1+1)^s\log(A_1+1)\right) \leq 2^s E(A_1^s\log2A_1) < \infty$ since our assumption $E(A_1^t)< \infty$ for each $0< t< 1$.

Therefore,
\[
\lim\limits_{s \to 0^+}\dfrac{\log E\left(X_1^{-s}\right)}{s} = \lim\limits_{s \to 0^+}\dfrac{-E\left(X_1^{-s}\log X_1\right)}{E\left(X_1^{-s}\right)} = -\int_{\Omega}\log X dP,
\]
where the equalities follow from the L'Hospital's rule. So we can choose $ 0< s_0 < 1$ such that
\begin{align}\label{fu zuihou 1}
\left|\dfrac{\log E\left(X_1^{-s_0}\right)}{s_0} + \int_{\Omega}\log X dP\right| < \frac{\delta}{8}.
\end{align}

Let $\varepsilon = \dfrac{1}{8}\delta s_0$. By Lemma 3, there exists $N_2 > 0$, such that for any integrable function $g \in \sigma(A_{N_2}, A_{N_2 + 1}, \cdots)$, we have
\begin{align}\label{fu zuihou 2}
E\left(X_1^{-s_0}|g|\right) \leq (1 + \varepsilon)^2E\left(X_1^{-s_0}\right)E\left(|g|\right).
\end{align}

Choose $\tau = \dfrac{s_0}{N_2}$. The following proof will be divided into two cases according to whether $n$ can be divided by $N_2$.

Case i. $n = kN_2$ for some $k \in \mathbb{N}$. With the help of Lemma 3, by the similar estimation for $E\left(X_1^{\lambda}X_2^{\lambda} \cdots X_n^{\lambda}\right)$ in the estimate of $\uppercase\expandafter{\romannumeral1}$, we obtain that
\[
E(X_1^{-\tau}X_2^{-\tau} \cdots X_n^{-\tau}) \leq \left(\left(E(X_1^{-s_0})\right)^n (1 + \varepsilon)^{2n} \right)^{\frac{1}{N_2}}.
\]

Therefore, write $\left((E\left(X_1^{-s_0}\right))^n (1 + \varepsilon)^{2n} \right)^{\frac{1}{N_2}}$ as the exponential form, we have
\begin{align*}
\uppercase\expandafter{\romannumeral2} \leq & \exp\left({n\tau(\int_{\Omega}\log X dP - \frac{\delta}{2})}\right)\left(E(\left(X_1^{-s_0}\right))^n (1 + \varepsilon)^{2n} \right)^{\frac{1}{N_2}}\\
= &\exp\left({n\tau(\int_{\Omega}\log X dP -\frac{\delta}{2} + \frac{\log E(X_1^{-s_0})}{s_0} + \frac{2}{s_0}\log (1 + \varepsilon))}\right).\notag
\end{align*}

Note that $\log(1 + x) \leq x$ for all $x \geq 0$, we know that
$$\frac{2}{s_0}\log (1 + \varepsilon) \leq \frac{2}{s_0}\varepsilon = \frac{\delta}{4},$$
by (\ref{fu zuihou 2}), we deduce that
\[
\int_{\Omega}\log X dP + \frac{\log E(X_1^{-s_0})}{s_0} -\frac{\delta}{2} + \frac{2}{s_0}\log (1 + \varepsilon)) \leq \frac{\delta}{8} - \frac{\delta}{2} + \frac{\delta}{4} = -\frac{\delta}{8}.
\]

Therefore,
\[
\uppercase\expandafter{\romannumeral2} \leq \exp\left({-\frac{1}{8}\tau\delta n}\right).
\]

Case ii. $n = kN_2 + l$ for some $k \in \mathbb{N}$ and $0 < l < N_2$. Since $X_i^{-\tau} \geq 1$ for $n+1 \leq i \leq (k + 1)N_2 $, we have
\begin{align}\label{zuihou}
E\left(X_1^{-\tau}X_2^{-\tau} \cdots X_n^{-\tau}\right) \leq E\left(X_1^{-\tau}X_2^{-\tau} \cdots X_{(k + 1)N_2}^{-\tau}\right).
\end{align}

Thus, by (\ref{fu Markov}), (\ref{zuihou}) and the result in Case i for $(k + 1)N_2$, we deduce that
\begin{align*}
\uppercase\expandafter{\romannumeral2} &\leq \exp\left({-\frac{1}{8}\tau\delta (k + 1)N_2}\right)\exp\left((l-N_2)\tau(\int_{\Omega}\log X dP -\frac{\delta}{2})\right)\\ \notag
& = \exp\left({-\frac{1}{8}\tau\delta n}\right)\exp\left((l-N_2)\tau(\int_{\Omega}\log X dP -\frac{3}{8}\delta)\right)
\end{align*}
Notice that $0 < l < N_2$, $\int_{\Omega}\log XdP < 0$ and $\tau= \dfrac{s_0}{N_2}$, we obtain that
\[
(l-N_2)\tau(\int_{\Omega}\log X dP -\frac{3}{8}\delta) \leq -s_0(\int_{\Omega}\log X dP -\frac{3}{8}\delta).
\]

Let $B_2 = \exp \left(-s_0(\int_{\Omega}\log X dP -\frac{3}{8}\delta)\right)$ and $\alpha_2 = \dfrac{1}{8}\tau\delta$. So for any $\delta >0$ and all $n \geq N_2$, we have
\[
\uppercase\expandafter{\romannumeral2} \leq B_2\exp(-\alpha_2 n).
\]

In conclusion, let $N = \max\{N_0, N_1, N_2\}$, $B = B_1 + B_2$ and $\alpha = \min\{\alpha_1, \alpha_2\}$, then for any $\delta > 0$ and all $n \geq N$, we have
\[
P\left(\left|\dfrac{1}{n}\log Q_n + \int_{\Omega}\log X(\omega)dP\right| \geq \delta\right) \leq B\exp(-\alpha n).
\]
\end{proof}

{\bf Acknowledgement}
Thank the referee for the helpful suggestion.
The work was supported by NSFC 11371148, 11201155, Guangdong Natural Science Foundation 2014A030313230 and \lq\lq the Science and Technology Development Fund of Macau (No. 069/2011/A)\rq\rq. N.-R. Shieh did this work while he visited SCUT in October 2013 and May 2014; the hospitality
is acknowledged. We thank Prof. Rui Kuang for his useful comments and discussions.


\begin{thebibliography}{10}
\bibitem{lesBryc92A} W. Bryc, {\it  On the large deviation principle for stationary weakly dependent random fields}, Ann. Probab. 20 (1992), no. 2, 1004-1030.

\bibitem{lesBryc92S} W. Bryc, {\it On large deviations for uniformly strong mixing sequences}, Stochastic Process. Appl. 41 (1992), no. 2, 191-202.



\bibitem{lesCT12} Y. Chung and H. Takahasi, {\it Large deviation principle for Benedicks-Carleson quadratic maps}, Comm. Math. Phys. 315 (2012), no. 3, 803-826.


\bibitem{lesCR11} H. Comman and J. Rivera-Letelier, {\it Large deviation principles for non-uniformly hyperbolic rational maps}, Ergodic Theory Dynam. Systems 31 (2011), no. 2, 321-349.



\bibitem{les79} K. Dajani and C. Kraaikamp, {\it Ergodic theory of numbers}, Mathematical Association of America, Washington, DC, 2002.

\bibitem{les75} R. Durrett, {\it Probability: theory and examples. Fourth edition}, Cambridge University Press, Cambridge, 2010.

\bibitem{les80} M. Einsiedler and T. Ward, {\it Ergodic theory with a view towards number theory}, Springer-Verlag, London, 2011.


\bibitem{les94} C. Faivre, {\it Distribution of L\'{e}vy constants for quadratic numbers}, Acta Arith. 61 (1992), 13-34.

\bibitem{les93} C. Faivre, {\it On the central limit theorem for random variables related to the continued fraction expansion}, Colloq. Math. 71 (1996), 153-159.

\bibitem{les88} C. Faivre, {\it The L\'{e}vy constant of an irrational number}, Acta Math. Hungar. 74 (1997), 57-61.

\bibitem{les91} M. Gordin and M. Reznik, {\it The law of the iterated logarithm for the denominators of continued fractions}, Vestnik Leningrad Univ. 25 (1970), 28-33.

\bibitem{les81} B. Jones and J. Thron, {\it Continued fractions: analytic theory and applications}, Cambridge University Press, Cambridge, 1980.

\bibitem{les76} Y. Khintchine, {\it Contidued fractions}, The University of Chicago Press, Chicago, 1964.

\bibitem{lesKif90} Y. Kifer,{\it Large deviations in dynamical systems and stochastic processes}, Trans. Amer. Math. Soc. 321 (1990), no. 2, 505-524.

\bibitem{les87} Y. Kifer, Y. Peres and B. Weiss, {\it A dimension gap for continued fractions with indenpendent digits}, Israel J. Math. 124 (2001), 61-76.

\bibitem{lesKif04} Y. Kifer,{\it Averaging principle for fully coupled dynamical systems and large deviations}, Ergodic Theory Dynam. Systems 24 (2004), no. 3, 847-871.


\bibitem{LP12} G. Letac and M. Piccioni, {\it Random continued fractions with beta-hypergeometric distribution},  Ann. Probab. 40 (2012), no. 3, 1105-1134.

\bibitem{les82} P. L\'{e}vy, {\it Sur les lois de probabilit\'{e} dont d\'{e}pendent les quotients complets et incomplets d\'{u}ne fraction continue},
    Bull. Soc. Math. 57 (1929), 178-194.

\bibitem{Lorentzen13} L. Lorentzen, {\it Limiting behavior of random continued fractions}, Constr. Approx. 38 (2013), no. 2, 171-191.

\bibitem{Lyons00} R. Lyons, {\it Singularity of some random continued fractions}, J. Theoret. Probab. 13 (2000), no. 2, 535-545.

\bibitem{MN08} I. Melbourne and M. Nicol, {\it Large deviations for nonuniformly hyperbolic systems}, Trans. Amer. Math. Soc. 360 (2008), no. 12, 6661-6676.


\bibitem{Mel09} I. Melbourne , {\it Large and moderate deviations for slowly mixing dynamical systems}, Proc. Amer. Math. Soc. 137 (2009), no. 5, 1735-1741.

\bibitem{les92} G. Misevi\v{c}ius, {\it Estimate of the remainder term in the limit theorem for denominators of continued fractions}, Lithuanian Math. J. 21 (1981), 245-253.


\bibitem{lesOP88} S. Orey and S. Pelikan, {\it Large deviation principles for stationary processes}, Ann. Probab. 16 (1988), no. 4, 1481-1495.

\bibitem{lesOP89} S. Orey and S. Pelikan, {\it Deviations of trajectory averages and the defect in Pesin's formula for Anosov diffeomorphisms}, Trans. Amer. Math. Soc. 315 (1989), no. 2, 741-753.


\bibitem{les83} W. Philipp, {\it Some metrical theorems in number theory}, Pacific J. Math. 20 (1967), 109-127.

\bibitem{les89} W. Philipp and P. Stackelberg, {\it Zwei Grenzwerts\"{a}tze f\"{u}r Kettenbr\"{u}che}, Math. Ann. 181 (1969), 152-154.

\bibitem{les85} W. Philipp, {\it Some metrical theorems in number theory II}, Duke Math. J. 37 (1970), 447-458.

\bibitem{les84} W. Philipp, {\it Limit theorems for sums of partial quotients of continued fractions}, Monatsh. Math. 105 (1988), 195-206.

\bibitem{lesRY08} L. Rey-Bellet and L. Young, {\it Large deviations in non-uniformly hyperbolic dynamical systems}, Ergodic Theory Dynam. Systems 28 (2008), no. 2, 587-612.

\bibitem{les78} A. Shiryaev, {\it Probability}, Second edition, Springer-Verlag, New York, 1996.

\bibitem{SSU01} K. Simon, B. Solomyak and M. Urba\'{n}ski, {\it Invariant measures for parabolic IFS with overlaps and random continued fractions}, Trans. Amer. Math. Soc. 353 (2001), no. 12, 5145-5164 (electronic).

\bibitem{les90} P. Stackelberg, {\it On the law of the iterated logarithm for continued fractions}, Duke Math. J. 33 (1966), 801-820.

\bibitem{les77} P. Walters, {\it An introduction to ergodic theory}, Springer-Verlag, New York-Berlin, 1982.

\bibitem{les86} J. Wu, {\it A remark on the growth of the denominators of convergents}, Monatsh. Math. 147 (2006), 259-264.

\bibitem{les95} J. Wu, {\it On the L\'{e}vy constants for quadratic irrationals}, Proc. Amer. Math. Soc. 134 (2006), 1631-1634.

\bibitem{les96} J. Wu, {\it Continued fraction and decimal expansions of an irrational number}, Adv. Math. 206 (2006), 684-694.

 \bibitem{lesYou90} L. Young, {\it Large deviations in dynamical systems}, Trans. Amer. Math. Soc. 318 (1990), no. 2, 525-543.

\end{thebibliography}
\end{document}